\documentclass[12pt]{amsart}
\usepackage{amsfonts}
\usepackage{amsmath}
\usepackage{amssymb}
\usepackage[margin=1.2in]{geometry}
\usepackage{hyperref}
\newcommand{\R}{\mathbb{R}}

\newtheorem{theorem}{Theorem}[section]

\newtheorem{corollary}[theorem]{Corollary}
\newtheorem{lemma}[theorem]{Lemma}

\theoremstyle{definition}

\newtheorem{remark}[theorem]{Remark}
\newtheorem{claim}[theorem]{Claim}

\makeatletter
\@namedef{subjclassname@2020}{
  \textup{2020} Mathematics Subject Classification}
\makeatother

\numberwithin{equation}{section}
\numberwithin{theorem}{section}

\numberwithin{equation}{section}

\begin{document}
\title[Generalizations of the Wiener-Ikehara theorem]{Generalizations of Koga's version of the Wiener-Ikehara theorem}

\author[B. Chen]{Bin Chen}
\thanks{B. Chen gratefully acknowledges support by the China Scholarship Council (CSC) and Ghent University through a BOF postdoctoral fellowship}
\address{B. Chen\\ Department of Mathematics: Analysis, Logic and Discrete Mathematics\\ Ghent University\\ Krijgslaan 297\\ B 9000 Ghent\\ Belgium}
\email{bin.chen@UGent.be}
\author[J. Vindas]{Jasson Vindas}
\thanks{The work of J. Vindas was supported by the Research Foundation--Flanders, through the FWO-grant number G067621N, and by Ghent University, through the grant number bof/baf/4y/2024/01/155 }
\address{J. Vindas\\ Department of Mathematics: Analysis, Logic and Discrete Mathematics\\ Ghent University\\ Krijgslaan 297\\ 9000 Ghent\\ Belgium}
\email{jasson.vindas@UGent.be}

\begin{abstract} We establish new versions of the Wiener-Ikehara theorem where only boundary assumptions on the real part of the Laplace transform are imposed. Our results generalize and improve a recent theorem of T.~Koga [J. Fourier Anal. Appl. 27 (2021), Article No. 18]. As an application, we give a quick Tauberian proof of Blackwell's renewal theorem.
\end{abstract}

\maketitle

\section{Introduction}
The Wiener-Ikehara theorem \cite{Wienerbook} is a foundational result in complex Tauberian theory. Originally devised to significantly simplify an early result of Landau \cite{Landaubook1953} and so deliver one of the quickest deductions of the prime number theorem, it has found countless applications in diverse areas of mathematics such as operator theory, partial differential equations, and number theory. The interested reader is referred to the books \cite[Chapter 10]{diamond-zhangbook}, \cite[Chapter III]{korevaarbook}, and \cite[Chapter II.7]{TenenbaumBook} for excellent accounts on the subject and the recent articles \cite{B-D-V2021,D-V2016,Debruyne-VindasCT,FT2019,Koga2021, MSJ2024, revesz-roton,T-V,zhang2019}  for some developments during the last decade. See also \cite{B-B-T2016,C-S2016,D-V2018,kunisada2025,Seifert2016} for closely related complex Tauberian theorems for Laplace transforms, such as the Ingham-Karamata theorem.

In one of its many forms, the Wiener-Ikehara theorem states that if a non-decreasing function $S$ has convergent Laplace transform $\mathcal{L}\{S;s\} = \int^{\infty}_{0} e^{-sx}S(x) \mathrm{d}x$ for  $\Re e \: s > 1$ and if there is a constant $a\in\mathbb{R}$ such that the analytic function
\begin{equation}
\label{KWI eq pole function}
G(s)=\mathcal{L}\{S;s\}-\frac{a}{s-1}
\end{equation}
admits $L^{1}_{loc}$-boundary behavior on the whole boundary line $1+i\mathbb{R}$, then 
\begin{equation} \label{conclusion W-I}
		S(x) \sim ae^{x} \qquad \mbox{as }x\to \infty.
	\end{equation}
Naturally, the hypothesis of $L^{1}_{loc}$-boundary behavior covers the case of continuous extension, and in particular that of analytic continuation. On the other hand, we point out that the boundary requirements on the Laplace transform can further be taken to a minimum if one employs the so-called local pseudofunction boundary behavior (cf. \cite{D-V2016,Debruyne-VindasCT,korevaar2005}). The pseudofunction approach plays a major  role in modern complex Tauberian theory (cf. \cite[Chapter III]{korevaarbook}).
	
Very recently \cite{Koga2021}, Koga has obtained an interesting generalization of this version of the Wiener-Ikehara theorem with $L^{1}_{loc}$-boundary behavior, where only the boundary properties of the real part of the Laplace transform are needed. His result also weakens the non-decreasing hypothesis on $S$ to log-linear slow decrease, a Tauberian condition that was introduced and studied in \cite{D-V2016,zhang2019} and that is intimately connected with exact Wiener-Ikehara theorems, that is, complete Laplace transform characterizations of the asymptotic behavior \eqref{conclusion W-I}. We call a function $S$ \emph{log-linearly slowly decreasing} (at $\infty$) if for each $\varepsilon$ there are $h,x_0>0$ such that
\[
\frac{S(y) - S(x)}{e^{x}}\geq -\varepsilon \qquad \mbox{for } x \leq y \leq x+h \mbox{ and } x\geq x_{0}.
\]

Koga's main motivation to establish a novel version of the Wiener-Ikehara theorem was to provide a Dirichlet series generalization of the Kolmogorov-Erd\H{o}s-Feller-Pollard renewal theorem \cite{EFP1949,Kolmogorov1936} (cf. \cite[Sections XIII.3 and XIII.11]{FellerBook1}). Moreover, he also obtained a Tauberian theorem for power series and applied it to give a new proof of the classical quoted renewal theorem.   Upon a minor reformulation (cf. Remark \ref{KWI rm 6.2}), Koga's Tauberian theorem for Laplace transforms reads:

\begin{theorem}[{\cite[Theorem 2]{Koga2021}}] \label{thKoga} Let $S\in L^{1}_{loc}[0,\infty)$ be log-linearly slowly decreasing and satisfy 
\begin{equation}\label{Koga condition 1}
	\int_{1}^{\infty}\frac{|S(x)|}{x^2 e^{x}}\:\mathrm{d}x<\infty. 
\end{equation}	
Let $U(s)=\Re e\:\mathcal{L}\{S;s\}$. Assume there are $\lambda>0$ and $g\in L^{1}(-\lambda,\lambda)$ such that
\begin{equation}
\label{KWI eq bd}
U(\sigma+it)\geq g(t), \qquad \mbox{for a.e. } t\in (-\lambda,\lambda) \mbox{ and } \sigma\in (1,2].
\end{equation}
If in addition $U$ has $L^{1}_{loc}$-boundary behavior on the boundary open subset $1+i(\mathbb{R}\setminus\{0\})$, namely, if there is $f\in L^{1}_{loc}(\mathbb{R}\setminus\{0\})$ such that on any finite closed interval $I$ not containing the origin we have
\begin{equation}
\label{L^1 boundary eq}
\lim_{\sigma\to1^{+}} \int_{I} |U(\sigma+it)-f(t)|\mathrm{d}t=0,
\end{equation}
 then \eqref{conclusion W-I} must hold for some constant $a\in\mathbb{R}$.
\end{theorem}

The aim of this paper is to considerably improve Koga's theorem by showing that it still holds true if one removes condition \eqref{Koga condition 1} from its set of hypotheses. In addition to
hold under weaker assumptions, we shall also consider a new useful alternative hypothesis for the boundary behavior of $\Re e\:\mathcal{L}\{S;s\}$ near $s=1$.

\begin{theorem}[Laplace transforms] \label{KWIthmain} Let $S\in L^{1}_{loc}[0,\infty)$ be  log-linearly slowly decreasing and have convergent Laplace transform on $\Re e \: s > 1.$
Suppose that the harmonic function $U(s) = \Re e \: \mathcal{L}\{S;s\}$ has $L^{1}_{loc}$-boundary behavior on $1+i(\mathbb{R}\setminus\{0\})$ and that there is some $\lambda>0$ such that one of the following two conditions holds:
\begin{itemize}
\item [(B.1)] there is $g\in L^{1}(-\lambda,\lambda)$  such that \eqref{KWI eq bd} holds;
\item [(B.2)] $\displaystyle \sup_{1<\sigma<2}\int_{-\lambda}^{\lambda} |U(\sigma+it)| \mathrm{d}t<\infty$.

\end{itemize}
Then
\begin{equation} 
\label{KWI conclusion}
	S(x) \sim ae^{x} \qquad \mbox{as }x\to\infty,
\end{equation}
where $a\in \mathbb{R}$ is in fact given by
\begin{equation} 
\label{KWI eq limit Laplace}
	a=\lim_{\sigma\to{1}^{+}} (\sigma-1)U(\sigma).
\end{equation}
\end{theorem}

When $S$ is non-decreasing, it is clearly automatically log-linearly slowly decreasing. In this case however, it is more natural to work with its Laplace-Stieltjes transform 
$\mathcal{L}\{\mathrm{d}S;s\}=\int_{0^{-}}^{\infty}e^{-s x}\mathrm{d}S(x)$
instead of the Laplace transform of the function. We shall show the following version of our Tauberian theorem for Laplace-Stieltjes transforms. 

\begin{theorem}[Laplace-Stieltjes transforms]\label{KWIthStieltjes} Let $S$ be log-linearly decreasing and of local bounded variation on $[0,\infty)$ with convergent Laplace-Stieltjes transform on $\Re e\: s>1$. Suppose that the hypotheses of Theorem \ref{KWIthmain} are satisfied with $U(s)=\Re e\: \mathcal{L}\{\mathrm{d}S; s\}$ instead of $\Re e\: \mathcal{L}\{S; s\}$. Then \eqref{KWI conclusion} and \eqref{KWI eq limit Laplace} still hold true.
\end{theorem}

Working with this new formulation has great practical value as in certain situations it is easier to apply than Theorem \ref{KWIthmain}. In fact, we will deduce the following corollary of Ingham-Karamata type from Theorem \ref{KWIthStieltjes}. We shall exemplify its usefulness in Section \ref{Section Blackwell's theorem} by giving a quick Tauberian proof of Blackwell's renewal theorem \cite{Blackwell}. We also give there a simpler treatment of Koga's renewal theorem for Dirichlet series {\cite[Theorem 5]{Koga2021} based on Theorem \ref{KWIthStieltjes}.

\begin{corollary}[Ingham-Karamata type theorem]\label{KWIthIK} Let $T$ be a function on $[0,\infty)$ such that $T(x)+Mx$ is non-decreasing, for some constant $M>0$, and such that its Laplace-Stieltjes transform converges on $\Re e\: s>0$. Suppose that  $U(s)=\Re e\: \int_{0^{-}}^{\infty} e^{-sx}\mathrm{d}T(x)$  has $L^{1}_{loc}$-boundary behavior on $i\mathbb{R}\setminus\{0\}$ and that there is some $\lambda>0$ such that one of the following two conditions holds:
\begin{itemize}
\item [(B$_{0}$.1)] there is $g\in L^{1}(-\lambda,\lambda)$ such that $U(\sigma+it)\geq g(t)$ for a.e.  $t\in (-\lambda,\lambda)$ and  $\sigma\in (0,1]$;
\item [(B$_{0}$.2)] $\displaystyle \sup_{0<\sigma<1}\int_{-\lambda}^{\lambda} |U(\sigma+it)| \mathrm{d}t<\infty$.
\end{itemize}
Then, there is a function $\tau(x)=o(1)$ as $x\to\infty$ such that
\begin{equation} 
\label{KWI conclusion IK}
	T(x)= ax + \int_{0}^{x} \tau(u) \mathrm{d}u + o(1) \qquad \mbox{as }x\to\infty, 
\end{equation}
where $a$ is given by
\begin{equation} 
\label{KWI eq limit Laplace IK}
	a=\lim_{\sigma\to{0}^{+}} \sigma U(\sigma).
\end{equation}

\end{corollary}

We shall also prove the next Tauberian theorem for power series, which improves upon \cite[Theorem 3]{Koga2021}.

\begin{theorem}[Power series]\label{KWIthpowerseries} Let   $F(z)=\sum_{n=0}^{\infty} c_n z^n$ be analytic on the unit disc $\mathbb{D}$ with real coefficients $\{c_n\}_{n=0}^{\infty}$. Suppose that the harmonic function $U(z)=\Re e\: F(z)$ has $L^{1}_{loc}$-boundary behavior on $\partial\mathbb{D}\setminus\{1\}$ and there is some $\theta_0\in (0,\pi)$ such that one of the following two conditions holds:
\begin{itemize}
\item [(b.1)] there is $g\in L^{1}(-\theta_0,\theta_0)$ such that $
U(re^{i \theta})\geq g(\theta)$ for a.e. $\theta\in (-\theta_0,\theta_0) $ and  $r\in [0,1)$;

\item [(b.2)] $\displaystyle \sup_{0<r<1}\int_{-\theta_0}^{\theta_0} |U(re^{\theta})| \mathrm{d}\theta<\infty$.
\end{itemize}
Then $\{c_n\}_{n=0}^{\infty}$ is convergent. In particular, its limit is given by 
\begin{equation}
\label{KWI eq limit}
\lim_{n\to\infty}c_n=\lim_{r\to 1^{-}}(1-r)U(r).
\end{equation}
\end{theorem}

The plan of the article is as follows. We discuss in Section \ref{Section Blackwell's theorem}  how Theorem \ref{KWIthStieltjes}, Corollary \ref{KWIthIK}, and Theorem \ref{KWIthpowerseries} can be applied  to renewal theory \cite[Chapter XI]{FellerBook2}; our applications emphasize the role of the assumptions (B.1), (B$_{0}$.1), and (b.1) in the corresponding cases, which make the theorems relatively simple to apply. In Section \ref{KWI Tauberian lemmas}, we obtain a slight extension of the exact Wiener-Ikehara Tauberian theorem \cite[Theorem 3.6]{D-V2016}, where we shall show that \eqref{conclusion W-I} holds if and only if $S$ is log-linearly slowly decreasing, its Laplace transform converges for $\Re e\: s>1$, and the real part of the function $G$ given by \eqref{KWI eq pole function} has so-called local pseudofunction boundary behavior on $1+i\mathbb{R}$. Section \ref{KWI section proof Laplace transform versions} is devoted to the proofs of Theorem \ref{KWIthmain} and Theorem \ref{KWIthStieltjes}; our approach there will be to reduce them to the exact Wiener-Ikehara theorem from Section \ref{KWI Tauberian lemmas}. Theorem \ref{KWIthpowerseries} will be shown in Section \ref{KWI power series section}, while a proof of Corollary \ref{KWIthIK} will be given in Section \ref{Section proof of corollary}. Finally, we close the article with some remarks and further extensions of our Tauberian theorems, which will be discussed in Section \ref{section KWI concluding remarks}.

\section{Application: Renewal theorems}
\label{Section Blackwell's theorem}

Before showing our new versions of the Wiener-Ikehara theorem, we illustrate their usefulness with some applications. Our first application is to probability theory. We will give in this section a quick simple Tauberian proof of a fundamental result in renewal theory, namely, the renewal theorem \cite{FellerBook2}.  

Let $\mathrm{d}P$ be a probability measure\footnote{All measures considered in this article are locally finite Borel measures and their primitives are normalized to be right continuous and supported on the same interval as the measure when applicable.} on $[0,\infty)$  that is continuous at the origin, namely, $P(0)=0$. Its renewal function $Q$ is determined by the convolution equation
\begin{equation}\label{KWI eq 2.1}
\mathrm{d}Q= \delta+\mathrm{d}Q\ast\mathrm{d}P,
\end{equation} where hereafter $\delta$ stands for the Dirac delta measure concentrated at $0$.
 In fact,  the solution to  \eqref{KWI eq 2.1} is given by the convergent\footnote{Unlike $\mathrm{d}P$, the measure $\mathrm{d}Q$ might not be finite, the convergence is thus interpreted in e.g. the space of Radon measures.} series $\mathrm{d}Q=\sum_{n=0}^{\infty}\mathrm{d}P^{\ast n}$. 
 
 We shall distinguish two cases for $\mathrm{d}P$. We say that it is \emph{lattice} if there is $\alpha>0$ such that $\mathrm{d}P$ is concentrated on $\alpha \mathbb{N}=\{\alpha, 2\alpha , 3\alpha \dots \}$ (when $\alpha$ is maximal we call it its \emph{span}); otherwise, we shall call $\mathrm{d}P$ \emph{non-lattice}. We mention that the expectation of the probability measure $\mathrm{d}P$ may or may not be finite, and in the latter case the limit in \eqref{KWI eq 2.2} is interpreted as 0.
 
 \begin{theorem}[The renewal theorem \cite{Blackwell,EFP1949,Kolmogorov1936}]
If $\mathrm{d}P$ is non-lattice, then, for each $h>0$,
\begin{equation}
\label{KWI eq 2.2} Q(h+x)-Q(x)\to \frac{h}{\int_{0}^{\infty}x\:\mathrm{d}P(x)}, \qquad x\to \infty.
\end{equation}
For lattice $\mathrm{d}P$ with span $\alpha>0$, the relation  \eqref{KWI eq 2.2} holds for all $h=n \alpha$, $n\in\mathbb{N}$.
\end{theorem}
\begin{proof} We divide the proof into the corresponding two cases.

\emph{Non-lattice $\mathrm{d}P$} (Blackwell's renewal theorem). Let $F(s):=\mathcal{L}\{\mathrm{d}Q; s\}$. Laplace transforming \eqref{KWI eq 2.1}, we obtain
\begin{equation}
\label{KWI eq 2.3} F(s)= \frac{1}{1-G(s)}, \qquad \Re e\: s>0,
\end{equation}
with $G(s)=\mathcal{L}\{\mathrm{d}P; s\}$. The function $G(s)$ clearly extends continuously to the imaginary axis $i\mathbb{R}$ (because $\mathrm{d}P$ is a finite measure), and, with the exception of $s=0$, we have $G(s)\neq 1$ for all other points of $\{s:\Re e\: s\geq 0\}$ (since otherwise $\mathrm{d}P$ would necessarily be lattice). We conclude that $F$ has a continuous extension to $i\mathbb{R}\setminus\{0\}$ and in particular has $L^{1}_{loc}$-behavior on this boundary subset. Furthermore, 
$$
\Re e\: F(s)= \frac{1-\int_{0^{-}}^{\infty}e^{-\sigma x}\cos (t x)\mathrm{d}P(x)}{|1-G(s)|^{2}}>0, \qquad \sigma=\Re e \:s>0.
$$
Corollary \ref{KWIthIK} (version with assumption (B$_{0}$.1)) applied to the non-decreasing function $Q$ then yields 
$$ Q(x)= ax+ \int_{0}^{x}\tau(u)\mathrm{d}u +o(1)$$
with (see \eqref{KWI eq limit Laplace IK})
$$
a=\lim_{\sigma\to 0^{+}} \sigma F(\sigma)=\lim_{\sigma\to 0^{+}}\frac{\sigma} {1-\int_{0^{-}}^{\infty} e^{-\sigma x}\mathrm{d}P(x) }= \frac{1}{\int_{0}^{\infty}x\mathrm{d}P(x)}
$$
 and some function $\tau(x)=o(1)$,  whence \eqref{KWI eq 2.2} follows at once.

\emph{Lattice $\mathrm{d}P$} (the Kolmogorov-Erd\H{o}s-Feller-Pollard renewal theorem). In this case $\mathrm{d} P(x)=\sum_{n=0}^{\infty} p_{n} \delta(x-n\alpha)$ and $\mathrm{d} Q(x)=\sum_{n=0}^{\infty} q_{n} \delta(x-n\alpha)$
with $q_0=1$,  $p_{0}=0$, and $\sum_{n=1}^{\infty}p_n=1$. Furthermore, these non-negative sequences are linked by the convolution relation
\begin{equation}
\label{KWI eq 2.4}
q_n=\sum_{k=1}^{n} p_{k}q_{n-k}, \qquad n\geq 1. 
\end{equation}
Since we assumed $\alpha$ to be maximal, we have $1=\mathrm{gcd}\{n: p_n\neq 0\}$, which implies that $G(re^{i\theta})\neq 1$ for all $\theta\in [-\pi,\pi] \setminus\{0\}$. Here $G$ stands for the power series $G(z)=\sum_{n=1}^{\infty} p_n z^{n}$, which is continuous on the closed unit disc. Due to \eqref{KWI eq 2.4}, we obtain $F(z)=\sum_{n=0}^{\infty}q_nz^n=(1-G(z))^{-1}$.   As in the previous the case, we also have  $\Re e\: F(z)>0$ for all $z\in\mathbb{D}$ (since $\Re e\: G(z)<1$ on $\mathbb{D}$). Hence, Theorem \ref{KWIthpowerseries} allows us to conclude that 
$$
\lim_{n\to\infty} q_n=\lim_{r\to1^{-}} \frac{1-r}{1-\sum_{n=1}^{\infty}r^{n}p_n} = \frac{1}{\sum_{n=1}^{\infty}np_n},
$$
which completes the proof of the renewal theorem.
\end{proof}

We can also give a simpler proof than Koga's original one for his version of the renewal theorem for Dirichlet series. The symbol $\star$ below stands for the Dirichlet convolution \cite{TenenbaumBook} of two arithmetic functions, while $e$ denotes the identity of this convolution, namely, the arithmetic function $e:\mathbb{N}\to\mathbb{R}$ given by $e(1)=1$ and $e(n)=0$ for $n\geq2$.

\begin{theorem}[{\cite[Theorem 5]{Koga2021}}]
Let $g:\mathbb{N}\to [0,\infty)$ be such that  $\sum_{n=2}^{\infty}{g(n)}/n=1$, $g(1)=0$, and no set $\{d^{k}: k\in\mathbb{N}\}$ with $d\geq2$ entirely contains $\{n: \: g(n)\neq0\}$. If $f$ is defined through $f=e+ f\star g$, then 
$$
\lim_{x\to\infty}\frac{1}{x}\sum_{n\leq x} f(n)= \displaystyle\frac{1}{\sum_{n=2}^{\infty} \frac{g(n)\log n}{n}}\:.
$$
\end{theorem}
\begin{proof}
We set $S(x)=\sum_{n\leq e^{x}}f(n)$, so that $F(s)=\mathcal{L}\{\mathrm{d}S;s\}=\sum_{n=1}^{\infty}f(n)/n^{s}$. The familiar properties of Dirichlet series and $f=e+ f\star g$ yield \eqref{KWI eq 2.3} with now $G$ given by the Dirichlet series of $g$, i.e., $G(s)=\sum_{n=2}^{\infty}g(n)/n^{s}$, which continuously extends to $\Re e\:s=1$. In view of \cite[Lemma 11]{Koga2021}, the assumption on $\{n: \: g(n)\neq0\}$ implies that $G(1+it)\neq 1$ for $t\neq0$. Obviously, $\Re e \: F(s)> 0$ on the entire open half-plane $\Re e\: s>1$. An application of Theorem \ref{KWIthStieltjes} thus shows that
$$
\lim_{x\to\infty}\frac{ S(x)}{e^{x}} =\lim_{\sigma\to1^{+}} (\sigma-1)F(\sigma)= \lim_{\sigma\to1^{+}} \frac{(\sigma-1)}{1-G(\sigma)}= \displaystyle\frac{1}{\sum_{n=2}^{\infty} \frac{g(n)\log n}{n}}\:.
$$
\end{proof}

\section{Exact Wiener-Ikehara theorem revisited}
\label{KWI Tauberian lemmas}
One of the exact Wiener-Ikehara theorems from \cite{D-V2016} states that for $S\in L^{1}_{loc}[0,\infty)$ to satisfy $S(x)\sim a e^{x}$ is necessary and sufficient that $S$ is log-linearly slowly decreasing, its Laplace transform converges for $\Re e\: s>1$, and 
\[G(s)=\mathcal{L}\{S;s\}-\frac{a}{s-1}\]
 has local pseudofunction boundary behavior on $1+i\mathbb{R}$. We wish to replace $G$ by its real part in this characterization. This will be done in fact in Corollary \ref{KWI exact theorem} below, but before we move on, let us briefly recall what is meant by local pseudofunction boundary behavior.
  
 In this and the next sections, we shall make use of Schwartz distribution theory. Our notation for calculus with distributions is as in the standard textbooks \cite{vladimirovbook} or \cite{estrada-kanwalBook}; in particular, we make use of dummy variables of evaluation to facilitate our manipulations. 
 As usual, $\mathcal{D}(I)$ stands for the space of smooth test functions with compact supports on an open set $I\subset \mathbb{R}$, while $\mathcal{D}'(I)$ is the space of distributions on $I$. We say that $f\in\mathcal{D}'(I)$ is a local \emph{pseudofunction} if for every $\varphi\in\mathcal{D}(I)$ the (distributional) Fourier transform of $\varphi f$ (which is entire by the Paley-Wiener theorem) is a continuous function that vanishes at $\pm\infty$. We then write $f\in PF_{loc}(I)$. Note that $L^{1}_{loc}(I)\subset PF_{loc}(I)$, thanks to the Riemann-Lebesgue lemma. In what follows we exploit that $\mathcal{D}'$ and $PF_{loc}$ are both (fine) sheaves, which means we can work with localizations due to the existence of smooth partitions of unity \cite{vladimirovbook}. 
 
 Let $I\subset \mathbb{R}$ be open. A harmonic function $U$ on $\Re e\:s>1$ is said to have 
 distributional boundary values on $1+iI$ if there is $u\in\mathcal{D}'(I)$ such that
 $$
 \lim_{\sigma\to 1^{+}} U(\sigma+it)=u(t) \qquad \mbox{in } \mathcal{D}'(I),
 $$
 that is, if for each test function $\varphi\in \mathcal{D}(I)$, 
  $$
\lim_{\sigma\to 1^{+}}\int_{-\infty}^{\infty} U(\sigma+it) \varphi(t)\mathrm{d}t=\langle u(t),\varphi(t)\rangle. $$
We say that $U$ has local pseudofunction boundary behavior on $1+iI$ if it has distributional boundary values there and its boundary distribution $u\in PF_{loc}(I)$. We refer to \cite{estrada-kanwal1982,Langenbruch1978} for the theory of boundary values of harmonic functions in distribution spaces (the article \cite{Langenbruch1978} actually deals with the general case of distributional boundary values for zero solutions of partially hypoelliptic constant coefficient partial differential operators, such as the Laplacian in our case). We point out that the harmonic function $U$ has distributional boundary values on $1+iI$ if and only if for each compact $K\subset I$ one can find $k=k(K)$ such that 
\begin{equation}
U(\sigma+it)=O\left(\frac{1}{(\sigma-1)^{k}}\right)
\end{equation}
for $t\in K$ and, say, $1<\sigma<2$.

The next lemma is our most important technical tool in this section.

\begin{lemma}\label{KWI main lemma}  Let $U$ be a real-valued harmonic function on the half-plane $\Re e\: s>1$ and let $I\subseteq \mathbb{R}$ be open. If $U$ has local pseudofunction boundary behavior on the boundary set $1+iI$, so does any harmonic conjugate to $U$.
\end{lemma}
\begin{proof} $Let$ $V$ be a harmonic conjugate to $U$. Note that since $U$ has distributional boundary values, then $V$ should also admit a boundary distribution\footnote{This is easily seen for a harmonic function on the unit disc $U(re^{i\theta})=\sum_{n\in\mathbb{Z}} c_n r^{|n|} e^{in\theta}$, because having distributional boundary values in this case becomes equivalent to $\{c_n\}_{n\in\mathbb{Z}}$ being of at most polynomial growth (see e.g. \cite{estrada-kanwal1982}). Since our assertion is local, the general case follows by applying conformal maps mapping  boundary segments into disc arcs.}. It suffices to see that the analytic function $F=U+iV$ has local pseudofunction boundary behavior on $1+iI$. We first show this under the additional assumption $U(\bar{s})=U(s)$. Using the Cauchy-Riemann equations, we see that $(V(s)-V(\bar{s}))/2$ must also be harmonic conjugate to $U$. Therefore, dropping a constant summand, we may assume that $V$ satisfies $V(\bar{s})=-V(s)$. We might also assume that $I$ is symmetric about the origin. Set
$$
u(t)=\lim_{\sigma\to1^{+}} U(\sigma+it)\in {PF}_{loc}(I)
$$
and
$$f(t)=\lim_{\sigma\to 1^{+}} F(\sigma +it)\in \mathcal{D}'(I).
$$
Since $f$ is the boundary value of an analytic function,  \cite[Proposition 2.1]{D-V2016} implies that $\langle f(t),\varphi(t)e^{ih t}\rangle=o(1)$ as $h\to-\infty$, for each $\varphi\in \mathcal{D}(I)$. Our assumption is $\langle u(t),\varphi(t)e^{ih t}\rangle=o(1)$ as $|h|\to\infty$. We also notice that $u$ is an even distribution, while $f(t)-u(t)$ is odd. For $h>0$ and $\varphi\in \mathcal{D}(I)$ real-valued and even,
\begin{align*}
\langle f(t),\varphi(t)e^{ih t}\rangle&= \langle f(t),\varphi(t)(e^{ih t}+e^{-iht})\rangle +o(1)
\\
&= \langle u(t),\varphi(t)(e^{ih t}+e^{-iht})\rangle +o(1)
\\
&=o(1) \qquad \mbox{as }h\to \infty.
\end{align*}
Likewise, for $\varphi\in \mathcal{D}(I)$ real-valued and odd,
\begin{align*}
\langle f(t),\varphi(t)e^{ih t}\rangle
&=  \langle u(t),\varphi(t)(e^{ih t}-e^{-iht})\rangle +o(1)
\\
&
=o(1) \qquad  \mbox{as }h\to \infty.
\end{align*}
Decomposing an arbitrary test function into real and imaginary parts, and then each of them into the sum of their even and odd parts, we obtain that $\langle f(t),\varphi(t)e^{ih t}\rangle= o(1)$ as $|h|\to\infty$ for each $\varphi\in \mathcal{D}(I)$, namely, $f\in{PF}_{loc}(I).$

A small variant of the above argument also applies when $U$ satisfies $U(s)=-U(\bar{s})$. Finally, the general case follows from these two particular ones by writing $U(s)= (U(s)+U(\bar{s}))/2 + (U(s)-U(\bar{s}))/2$.
\end{proof}
Lemma \ref{KWI main lemma} and \cite[Theorem 3.6]{D-V2016} together thus yield:

\begin{corollary}\label{KWI exact theorem}
 Let $S\in L^{1}_{loc}[0,\infty)$. Then,
$ S(x) \sim ae^{x}$
holds if and only if $S$ is log-linearly slowly decreasing, its Laplace transform is convergent on $\Re e\: s>1$, and the harmonic function
\begin{equation}
\label{LaplacePoleq}
\Re e\:\left(
\mathcal{L}\{S;s\} - \frac{a}{s-1}\right)
\end{equation}
admits local pseudofunction boundary behavior on the whole line $\Re e \: s = 1$.
\end{corollary}
\begin{remark}
\label{KWI r1} Let $S$ be of local bounded variation, so that $\mathcal{L}\{\mathrm{d}S;s\}=s\mathcal{L}\{S;s\}$. Since smooth functions are multipliers for local pseudofunctions,   $\mathcal{L}\{S;s\} - a/(s-1)$ has local pseudofunction boundary behavior on a given boundary subset if and only if  $\mathcal{L}\{\mathrm{d}S;s\} - a/(s-1)$ does it. Employing Lemma \ref{KWI main lemma} once more, we might replace \eqref{LaplacePoleq} by the hypothesis that the real part of $\mathcal{L}\{\mathrm{d}S;s\} - a/(s-1)$ has local pseudofunction boundary behavior on $\Re e\:s=1$.
\end{remark}

\begin{remark}
\label{KWI rm l1 vs pf} Lemma \ref{KWI main lemma} highlights a key advantage of the local pseudofunction approach over $L^1_{loc}$-boundary behavior. In fact, it is well-known that if a harmonic function has $L^1_{loc}$-boundary behavior, the distributional boundary values of its harmonic conjugate functions do not necessarily belong to $L^1_{loc}$ (see e.g. \cite[p.~73]{Katznelsonbook})
\end{remark}

\section{Proof of Theorem \ref{KWIthmain} (and Theorem \ref{KWIthStieltjes})}
\label{KWI section proof Laplace transform versions}

We are now ready to show Theorem \ref{KWIthmain}. We will do so with the aid of Corollary \ref{KWI exact theorem}. Observe that it suffices to prove \eqref{KWI conclusion}, since once this is established \eqref{KWI eq limit Laplace} automatically holds by the familiar real Abelian result for Laplace transforms.

The first part of our proof consists in showing that any of our assumptions imply that $U$ admits a boundary distribution on $1+i\mathbb{R}$. This is actually our hypothesis away from the boundary point $s=1$, where we even have the stronger $L_{loc}^{1}$-boundary behavior. So, we must then still establish the existence of a boundary distribution in a boundary neighborhood of 1. We need some preparation for it.

Let $\Phi:\mathbb{D}\to\Omega$ be a conformal equivalence between the unit disc and a region $\Omega\subset\{s: \:  \Re e\: s>1\}$ whose boundary is a smooth (namely, $C^\infty$) Jordan curve that meets the line $\Re e\:s=1$ in a closed segment containing the interval $1+i(-\lambda,\lambda)$.
 Classical results from the theory of conformal maps (see \cite[Theorem 2.6, p.~24; Theorem 3.5, p.~48; Theorem 3.6, p.~49]{PommerenkeBook} guarantee that $\Phi$ extends to a smooth diffeomorphism $\Phi:\overline{\mathbb{D}}\to\overline{\Omega}$, and we may assume that $\Phi(1)=1.$ Moreover, if $\Phi(J)=1+i[-\lambda,\lambda]$, then $\Phi$ has analytic extension through an open circular arc containing $J$, as one infers from the Schwarz reflection principle for analytic arcs \cite[Theorem 4.8, p.~19]{ConwayII}. We consider the harmonic function $V=U\circ\Phi$.

\begin{claim}
\label{KWIClaim1} The function $V$ belongs to the harmonic Hardy space $h^{1}(\mathbb{D})$.
\end{claim}
\begin{proof} If (B.2) is satisfied, we directly get $V\in h^{1}(\mathbb{D})$,
as inferred from\footnote{Theorem 10.1 is only stated for analytic functions in \cite[p.~168]{durenBook}, but the proof given there applies to harmonic functions as well.}
\cite[Theorem~10.1, p.~168]{durenBook}  and  \cite[Theorem 3.5, p.~48]{PommerenkeBook}. Assume now that (B.1) holds. Applying again \cite[Theorem~10.1, p.~168]{durenBook}, we conclude that $V$ has $L^1_{loc}$-boundary behavior on any boundary arc that does not contain the boundary point $1$. Combining the latter fact with a.e. existence of the radial boundary limits of $V$, we obtain that $V$ belongs to the harmonic Hardy space $h^{1}(S)$ for any disc sector of the form
 $$
  S=\{z\in\mathbb{D}: \: \arg z \notin J'\} \qquad \mbox{with a closed subarc } 1\in J' \subsetneq J.
 $$  
 Let us fix such a sector $S$. We also write $J=\{e^{i\theta}: \theta\in [\alpha,\beta]\}$. 
 We have
 $$
 \sup_{0<r<1} \int_{[-\pi,\pi]\setminus{(\alpha,\beta)}} |V(re^{i\theta})|\mathrm{d}\theta<\infty,
 $$
because $V\in h^{1}(S)$. Therefore, we just need to bound the integral of $|V(re^{i\theta})|$ over $[\alpha,\beta]$ as $r\to1^{-}$. Using the mean value property of harmonic functions and condition (B.1), we obtain
\begin{align*}
\int_{\alpha}^{\beta} |V(re^{i\theta})|\mathrm{d\theta}&\leq \int_{\alpha}^{\beta} |g(\Im m\: (\Phi(re^{i\theta})))|\mathrm{d}\theta + \int_{\alpha}^{\beta}[ V(re^{i\theta})- g(\Im m \: (\Phi(re^{i\theta})))]\mathrm{d}\theta
\\
& \leq 2 \int_{\alpha}^{\beta} |g(\Im m\: (\Phi(re^{i\theta})))|\mathrm{d}\theta + 2\pi V(0) +  \int_{[-\pi,\pi]\setminus{(\alpha,\beta)}} |V(re^{i\theta})|\mathrm{d}\theta
\\
&= 2 \int_{\alpha}^{\beta} |g(\Phi(e^{i\theta}))|\mathrm{d}\theta + o(1)+ 2\pi V(0) +  \int_{[-\pi,\pi]\setminus{(\alpha,\beta)}} |V(re^{i\theta})|\mathrm{d}\theta
\\
&=O(1), \qquad r\to 1^{-},
\end{align*}
where in the third line we have used that $\Phi(re^{i\theta})\to \Phi(e^{i\theta})=\Im m\: (\Phi(e^{i\theta}))$ and $(\Phi(re^{i\theta}))'\to (\Phi(e^{i\theta}))'$ uniformly on $[\alpha,\beta]$. This shows that $V\in h^{1}(\mathbb{D})$, also under (B.1).
\end{proof}

We can now establish our original claim:
\begin{claim}
\label{KWIClaim2} $U$ admits a boundary distribution on a boundary neighborhood of 1.
\end{claim}
\begin{proof}
In view of  \cite[Theorem~1.1, p.~2]{durenBook} and Claim \ref{KWIClaim1}, the harmonic function $V$ is a Poisson-Stieltjes integral, whence we readily obtain the bound
\[
V(z)=O\left(\frac{1}{1-|z|}\right), \qquad |z|<1.
\]
Since $\Phi$ extends to a diffeomorphism between complex neighborhoods of $J$ and $1+i[-\lambda,\lambda]$, we also have 
\[
U(\sigma+it)=O\left(\frac{1}{\sigma-1}\right), \qquad  \sigma+it \in (1,2]\times [-\lambda,\lambda].
\]
The latter bound yields the claim (see Section \ref{KWI Tauberian lemmas}).
\end{proof}

We can now move to the second part of the proof. Let  
$$u(t)=\lim_{\sigma\to 1^{+}}U(\sigma+it) \qquad \mbox{in } \mathcal{D}'(\mathbb{R}).$$ By assumption $u=f$ on $\mathbb{R}\setminus\{0\}$ with $f\in L^{1}_{loc}(\mathbb{R}\setminus\{0\})$. If (B.1) holds, then $u-g$ should be a non-negative measure on $(-\lambda,\lambda)$. When (B.2) is satisfied, $u$ is a measure on $(-\lambda,\lambda)$, as follows from the Banach-Alaoglu theorem (or Helly's selection principle as is better known in the Lebesgue-Stieltjes measure context), because (B.2) tells us that $\{U(\sigma+ i\:\cdot\:)\: : 1<\sigma<2\}$ is bounded in the dual of $C[-\lambda,\lambda]$. Summarizing, in every case we have shown that the distribution $u$ is a Radon measure on $\mathbb{R}$. Using the Lebesgue decomposition of $u$ (Lebesgue-Radon-Nikodym theorem \cite[p.~121]{rudinRC1987}), we conclude that $f\in L^{1}_{loc}(\mathbb{R})$ and that $f$ is its absolutely continuous part, while its singular part must have point support  at 0. Hence, $u=f+\pi a \delta$ for some constant $a\in\mathbb{R}$, where as usual $\delta$ stands for the Dirac delta distribution. Finally, using the well-known formula (cf. \cite[Eq.~(2.17), p.~58]{estrada-kanwalBook}) 
\[\lim_{\sigma\to 1^{+}} \frac{1}{\sigma-1+it}= \frac{-i}{t-i0}= \pi\delta(t)-i\: \mathrm{p.v.}\left(\frac{1}{t}\right),\] we deduce from Corollary \ref{KWI exact theorem} that $S(x)\sim a e^{x}$ as $x\to\infty$, because \eqref{LaplacePoleq} has boundary value distribution $f\in L^{1}_{loc}(\mathbb{R})\subset PF_{loc}(\mathbb{R})$ . This completes the proof of Theorem \ref{KWIthmain}.

The proof of Theorem \ref{KWIthStieltjes} is exactly the same as the one we just gave, but now making use of Remark \ref{KWI r1}.

\section{The power series case: proof of Theorem \ref{KWIthpowerseries}}
\label{KWI power series section}

The proof of Theorem \ref{KWIthpowerseries} is similar to that of Theorem \ref{KWIthmain}, but simpler since we can avoid using Corollary \ref{KWI exact theorem} via a more direct argument. It suffices to show that $\{c_n\}_{n=0}^{\infty}$ converges to some finite limit, because then necessarily \eqref{KWI eq limit} should hold due to Abel's classical limit theorem for power series.  As in Section \ref{KWI section proof Laplace transform versions}, the assumptions imply that $U(re^{i\theta})$ converges distributionally to a boundary measure, which is absolutely continuous with respect to the Lebesgue measure off $2\pi \mathbb{Z}$. Therefore, there are $f\in L^{1}[-\pi,\pi]$ and $a\in\mathbb{R}$ such that 
$$
\lim_{r\to{1}^{-}} U(re^{i\theta})= a\pi  \delta(\theta)+f(\theta)
$$
in, say, the dual of $C^\infty[-\pi,\pi].$
Now,
\begin{align*}
c_nr^{n}&= \frac{1}{2\pi} \int_{-\pi}^{\pi} F(r{e}^{i\theta}) e^{-in\theta}\mathrm{d}\theta
\\
&= \frac{1}{2\pi} \int_{-\pi}^{\pi} F(r{e}^{i\theta}) (e^{-in\theta}+e^{in\theta})\mathrm{d}\theta \\
&=\frac{1}{\pi} \int_{-\pi}^{\pi} U(r{e}^{i\theta}) \cos n\theta\: \mathrm{d}\theta,
\end{align*}
where we have used that $\{c_k\}_{k=0}^{\infty}$ is real.
Taking $r\to1^{-}$,
$$
c_n= a+\frac{1}{\pi} \int_{-\pi}^{\pi} f(\theta)\cos n\theta\: \mathrm{d}\theta= a+o(1) \qquad \mbox{as }n\to\infty,
$$
in view of the Riemann-Lebesgue lemma.

\section{Proof of Corollary \ref{KWIthIK}} \label{Section proof of corollary}

We set $S(x)=\int_{0^{-}}^{x}e^{u}\mathrm{d}T(u)$. Using that $T(x)+Mx$ is non-decreasing, we see that $S$ is log-linearly slowly decreasing; in fact, for $y\in[x,x+h]$,
\begin{align*}
\frac{S(y)-S(x)}{e^x}&= e^{-x}\int_{x^{+}}^{y} e^{u} (\mathrm{d}T(u)+ M\mathrm{d}u)- M e^{-x}\int_{x}^{y} e^u\mathrm{d}u
\\
&\geq M(1-e^{h})=o(1),
\end{align*}
 $h\to0^{+}$. The hypotheses (B$_0$.1) and (B$_0$.2) translate into the conditions (B.1) and (B.2), respectively, for the real part of the Laplace-Stieljes transform of $S$. Theorem \ref{KWIthStieltjes} then yields $S(x)\sim a e^{x}$ with $a$ given by \eqref{KWI eq limit Laplace IK}. Writing $\tau(x)=e^{-x}S(x)-a=o(1)$, noticing that $\mathrm{d}T(x)= e^{-x}\mathrm{d}S(x)$, and integrating by parts, we obtain,
\begin{align*}
T(x)&= a(x+1)+ \tau(x)+ \int_{0}^{x}\tau(u)\mathrm{d}u 
\\
&= ax +a +\int_{0}^{x}\tau(u)\mathrm{d}u+ o(1).
\end{align*}
The asymptotic formula \eqref{KWI conclusion IK} now follows upon redefining $\tau$ on a finite interval so that the constant $a$ gets absorbed into its integral on such an interval.

\section{Concluding remarks}\label{section KWI concluding remarks}
We end this article with some remarks.

\begin{remark}
\label{KWI rm 6.1} Theorem \ref{thKoga} is also directly covered by the version of Theorem \ref{KWIthmain} with the assumption (B.2). In fact, suppose that  \eqref{Koga condition 1} and \eqref{KWI eq bd}
 are satisfied. Let $\varphi$ be a real-valued even non-negative smooth function with support on $(-\lambda, \lambda)$ such that $\varphi(t)=1$ for $t\in[-\lambda/2,\lambda/2]$. Let $\widehat{\varphi}(x)=\int_{-\infty}^{\infty} \varphi(t)e^{-it x}\mathrm{d}x$, so that $\widehat{\varphi}$ is a Schwartz function. Then, since $U(\sigma+it)-g(t)\geq 0$ in the considered range,
\begin{align*}
\int_{-\lambda/2}^{\lambda/2} |U(\sigma+it )|\mathrm{d}t
&
\leq \int_{-\lambda/2}^{\lambda/2} |g(t)|\mathrm{d}t +\int_{-\infty}^{\infty} (U(\sigma+it )-g(t))\varphi(t)\mathrm{d}t
\\
&
\leq 2 \int_{-\infty}^{\infty} |g(t)|\varphi (t)\mathrm{d}t +\int_{-\infty}^{\infty} \mathcal{L}\{S;\sigma+it \}\varphi(t)\mathrm{d}t
\\
&
= 2 \int_{-\infty}^{\infty} |g(t)|\varphi (t)\mathrm{d}t+\int_{0}^{\infty} e^{-x}S(x)e^{-(\sigma-1)x}\widehat{\varphi}(x)\mathrm{d}x\\
&
\leq 2 \int_{-\infty}^{\infty} |g(t)|\varphi (t)\mathrm{d}t+ \int_{0}^{\infty} e^{-x}|S(x)||\widehat{\varphi}(x)|\mathrm{d}x<\infty.
\end{align*}
\end{remark}

\begin{remark}
\label{KWI rm 6.2} Koga originally stated his Tauberian theorem (cf. \cite[Theorem 2]{Koga2021}) by only imposing the boundary requirements for the Laplace transform on a sequence tending to $1^{+}$. More precisely, in addition to \eqref{Koga condition 1}, he assumes\footnote{His formulation of \eqref{L^1 boundary eq Koga} in \cite[Theorem 2]{Koga2021}  is slightly different, but equivalent in view of the well known completeness of the $L^1$-spaces.} the existence of $\sigma_n\to 1^{+}$ such that
\begin{equation}
\label{L^1 boundary eq Koga}
\lim_{n\to\infty} \int_{I} |U(\sigma_n+it)-f(t)|\mathrm{d}t=0, 
\end{equation}
for some  $f\in L^{1}_{loc}(\mathbb{R}\setminus\{0\})$ and any finite closed interval $I$ not containing the origin, and
\begin{equation}
\label{KWI eq bd Koga}
U(\sigma_n+it)\geq g(t), \qquad \mbox{for a.e. } t\in (-\lambda,\lambda) \mbox{ and } n\in\mathbb{N},\end{equation}
for some $\lambda>0$ and  $g\in L^{1}(-\lambda,\lambda) $. 

We have however  \eqref{L^1 boundary eq Koga} is equivalent to \eqref{L^1 boundary eq} in our case. In fact, since  \eqref{Koga condition 1}, which allows us to view $e^{-x}S(x)$ as a tempered distribution, ensures \cite[Section~6.6.9, p.~100]{vladimirovbook} 
that $\mathcal{L}\{S;s\}$ has distributional boundary values on $\Re e\: s>1$, the relation \eqref{L^1 boundary eq} might be inferred from \eqref{L^1 boundary eq Koga} by using a standard localization argument together with the (distributional) Schwarz reflection principle \cite{rudin1971} (see also \cite{Debrouwere-Vindas2023,Langenbruch1978} for generalized reflection principles). Likewise, the condition \eqref{KWI eq bd Koga} follows from  \eqref{KWI eq bd} and \eqref{Koga condition 1} (where one might need to subtract a constant from $g$ resulting from the application of the reflection principle). Alternatively, one can also see that Koga's original set of hypotheses is covered by  those of Theorem \ref{KWIthmain} via the same argument employed in Remark \ref{KWI rm 6.1}, which clearly yields, say, 
$$
\sup_{n\in\mathbb{N}}\int_{-2\lambda/3}^{2\lambda/3} |U(\sigma_n+it)|\mathrm{d}t<\infty.
$$
The latter condition in turn  implies (via localization and the reflection principle once more) that (B.2) holds with $\lambda$ replaced by, say, $\lambda/2$. 
\end{remark}

\begin{remark}
\label{KWI rm 6.3}
As Koga, we could also have only assumed in Theorem \ref{KWIthmain} and Theorem \ref{KWIthStieltjes} that \eqref{KWI eq bd} just holds on a sequence $\sigma=\sigma_n\to1^{+}$. Exactly the same argument given in Section \ref{KWI section proof Laplace transform versions} would then still yield that the boundary distribution of $U$ is a Radon measure. This comment also applies to Corollary \ref{KWIthIK} and Theorem \ref{KWIthpowerseries}.
\end{remark}

\begin{remark}
\label{KWI rm 6.4}
Sometimes one is just interested in deducing an upper bound
\begin{equation}
\label{KWI eq 1 bdd}
S(x)=O(e^x)
\end{equation}
from relatively mild regularity boundary properties of the Laplace transform. For instance, such criteria play an important role in abstract analytic number theory (see e.g. \cite[Chapter 11]{diamond-zhangbook} for applications to Beurling primes).
The following extension of \cite[Proposition 3.1]{D-V2016} (cf. \cite[Theorem 10.1]{diamond-zhangbook})  could then be useful in that respect. We call $S$ \emph{log-linearly boundedly decreasing} if there is $h>0$ such that
\begin{equation}
\label{eq:bdlog}
\liminf_{x\to\infty} \inf_{y\in [x,x+h]} \frac{S(y)-S(x)}{e^{x}}> -\infty,
\end{equation}
that is, 
if there are $x_0,M > 0$ such that
\begin{equation}
\label{deflogbsd}
 S(x+u) - S(x) \geq -Me^{x}, \quad \mbox{for } 0 \leq u \leq h \mbox{ and } x\geq x_{0}.
\end{equation}
We also use the notation $A_{c}(\mathbb{R})$ for the subspace of compactly supported elements of the Wiener algebra, namely, those compactly supported continuous functions such that their Fourier transforms belong to $L^{1}(\mathbb{R})$.

\begin{theorem}
\label{KWI boundedness proposition} Let $S\in L^{1}_{loc}[0,\infty)$ and let $\varphi\in A_{c}(\mathbb{R})\setminus\{0\}$ be even real-valued and have non-negative Fourier transform. Then, \eqref{KWI eq 1 bdd} holds if and only if $S$ is log-linearly boundedly decreasing, has convergent Laplace transform on $\Re e\: s>1$, and there is a sequence $\sigma_n\to1^{+}$ such that
\begin{equation}
\label{eq:b W-I general}
\mathfrak{I}_{\varphi}(h)=\lim_{n\to\infty} \int_{-\infty}^{\infty}\left( \Re e\mathcal{L}\{S;\sigma_n+it\}\right) \varphi(t)\cos ht\:\mathrm{d} t 
\end{equation}
exists for each $h>0$ and $\mathfrak{I}_{\varphi}(h)=O(1)$ as $h\to\infty$.
\end{theorem}

We will make use of the following lemma, which is basically shown within the proof of \cite[Proposition 3.1]{D-V2016}.

\begin{lemma}
\label{l: for boundedeness}
Let $S\in L^{1}_{loc}[0,\infty)$ be log-linearly boundedly decreasing and let $\psi\in L^{1}(\R)\setminus\{0\}$ be non-negative. If 
\[\int^{\infty}_{0} e^{-x}S(x)\psi(x-h)\mathrm{d}x=O(1) \qquad \mbox{as } h\to\infty,
\]
 then  $S(x)=O(e^{x})$ as $x\to\infty$.
\end{lemma}
\begin{proof}
Modifying $S$ in a finite interval, we may assume that \eqref{deflogbsd} holds for all $x\geq 0$. Iterating the inequality \eqref{deflogbsd}, one finds that there is $C$ such that 
\begin{equation} \label{eqlsditerated}
S(x) - S(y) \geq -Ce^{x} \quad \mbox{for all } x \geq y\geq 0.
\end{equation}
Set $B = \int^{\infty}_{0} e^{-x}\psi(x)\mathrm{d}x>0$. Using (\ref{eqlsditerated}), we obtain
\begin{align*}
 e^{-h}S(h) & = \frac{1}{B} \int^{\infty}_{0} e^{-x-h}S(h)\psi(x)\mathrm{d}x 
 \\
 &
 \leq \frac{1}{B}\int^{\infty}_{0}e^{-x-h}S(x+h)\psi(x) \mathrm{d}x + \frac{C}{B} \int^{\infty}_{0}\psi(x) \mathrm{d}x \\
& \leq \frac{1}{B}\int^{\infty}_{0} e^{-x}S(x)\psi(x-h)\mathrm{d}x + \frac{C}{B} \int^{\infty}_{-h}\psi(x) \mathrm{d}x+\frac{|S(0)|}{B}\int_{0}^{h}e^{-x}\psi(x-h)\mathrm{d}x 
\\
&
=O(1). 
\end{align*}
\end{proof}

\begin{proof}[Proof of Theorem \ref{KWI boundedness proposition}]  Set $\tau(x)=e^{-x}S(x)$ and $\widehat{\varphi}(x)=\int_{-\infty}^{\infty} \varphi(t)e^{-it x}\mathrm{d}x$.
 That the conditions are necessary is easy to verify. In fact, it is clear that $S$ must be log-linearly boundedly decreasing. Also,
 \begin{align*}\int_{-\infty}^{\infty} \Re e\mathcal{L}\{S;\sigma_n+it\} \: \varphi(t)\cos ht\: \mathrm{d} t
		&= \frac{1}{2}  \int_{-\infty}^{\infty}\mathcal{L}\{S;\sigma_n+it\} \varphi(t) \left(e^{iht}+e^{-iht}\right)\mathrm{d} t
		\\
		&= \frac{1}{2} \int_{0}^{\infty} e^{(1-\sigma_n) x}\tau(x) \left(\widehat{\varphi}(x-h)+\widehat{\varphi}(x+h)\right) \mathrm{d}x,
		\end{align*}
so that
\[
\mathfrak{I}_{\varphi}(h)=\frac{1}{2} \int_{0}^{\infty}\tau(x) \left(\widehat{\varphi}(x-h)+\widehat{\varphi}(x+h)\right) \mathrm{d}x.
\]
Hence,
\[
\limsup_{h\to\infty}|\mathfrak{I}_{\varphi}(h)| \leq \frac{1}{2}\:\|\widehat{\varphi}\|_{L^{1}(\R)} \limsup_{x\to\infty}|\tau(x)|.
\]

Conversely, let us now show the sufficiency of the conditions for \eqref{KWI eq 1 bdd}.  Some of the following arguments already occur in the proof of  \cite[Proposition 3.1]{D-V2016}. Since changing a function on a finite interval does not violate the assumption \eqref{eq:b W-I general}, we may assume that \eqref{eqlsditerated} holds. In addition, we may also thus suppose without loss of generality that $S$ is non-negative. In fact, if needed, one replaces $S$ by $S(x) - S(0) + Ce^{x}$, whose Laplace transform also satisfies \eqref{eq:b W-I general}. So, in what follows we work with the assumption $\tau\geq0$. Applying Lemma \ref{l: for boundedeness}, it suffices to establish that $(\tau \ast \widehat{\varphi})(h)=O(1)$ as $h\to\infty$. Thus (the use of Parseval's relation is justified by \cite[Lemma 3.4]{D-V2016})
\begin{align*}0\leq \int_{0}^{\infty} e^{(1-\sigma_n) x}\tau(x) \widehat{\varphi}(x-h)\mathrm{d}x&\leq \int_{0}^{\infty} e^{(1-\sigma_n) x}\tau(x) \left(\widehat{\varphi}(x-h)+\widehat{\varphi}(x+h)\right) \mathrm{d}x 
		\\
		&= \int_{-\infty}^{\infty}\mathcal{L}\{S;\sigma_n+it\} \varphi(t) \left(e^{iht}+e^{-iht}\right)\mathrm{d} t
		\\
		&= 2\int_{-\infty}^{\infty} \Re e\mathcal{L}\{S;\sigma_n+it\} \: \varphi(t)\cos ht\: \mathrm{d} t.
\end{align*}
The Beppo Levi theorem yields $
0\leq  (\tau \ast \widehat{\varphi})(h)\leq 2 \mathfrak{I}_{\varphi}(h)=O(1)$, $h\to\infty$.
\end{proof}
For example, Theorem \ref{KWI boundedness proposition} could have been used in Section \ref{KWI section proof Laplace transform versions} to directly show that $U(s)$ has distributional boundary values under (B.2) without having to pass through the conformal map argument, because once $\tau(x)=e^{-x}S(x)=O(1)$, the Laplace transform $\mathcal{L}\{S;s\}$ tends to the (distributional) Fourier transform of $\tau$ on $\Re e\: s=1$.
\end{remark}

\end{document}